\documentclass[english]{amsart}

\usepackage{esint}
\usepackage[svgnames]{xcolor} 
\usepackage[colorlinks,citecolor=red,pagebackref,
hypertexnames=false,breaklinks]{hyperref}
\usepackage{pgf,tikz}
\usepackage{pdfsync}

\usepackage{dsfont}
\usepackage{url}
\usepackage[utf8]{inputenc}
\usepackage[T1]{fontenc}
\usepackage{lmodern}
\usepackage{babel}
\usepackage{mathtools}  
\usepackage{amssymb}
\usepackage{lipsum}
\usepackage{mathrsfs}
\usepackage{color}

\usepackage[skip=2.2pt plus 1pt, indent=12pt]{parskip}

\newtheorem{theorem}{Theorem}[section]

\newtheorem{lemma}{Lemma}[section]

\numberwithin{equation}{section}

\title[Stability of determining a Dirichlet-Laplace-Beltrami]{Stability of 
determining a Dirichlet-Laplace-Beltrami operator from its boundary spectral 
data}

\author[Mourad Choulli]{Mourad Choulli}
\address{Universit\'e de Lorraine}
\email{mourad.choulli@univ-lorraine.fr}

\author[Masahiro Yamamoto]{Masahiro Yamamoto}
\address{Department of Mathematical Sciences, The University of Tokyo 3-8-1, 
Komaba, Meguro, Tokyo 153, Japan.
}
\email{myama@ms.u-tokyo.ac.jp}

\begin{document}

\begin{abstract}
We establish stability inequalities for the problem of determining 
a Dirichlet-Laplace-Beltrami operator from its boundary spectral data. 
We study  the case of complete spectral data as well as the case of partial 
spectral data.
\end{abstract}



\maketitle


\section{Stability in relationship with the elliptic Dirichlet-to-Neumann map}
\label{section1}

Let $\mathcal{M}$ be a $C^\infty$ connected compact manifold with boundary 
$\partial \mathcal{M}$ of dimension $n\ge 3$. We denote the set of metric 
tensors on $\mathcal{M}$ by $\mathscr{G}$.  Here by a metric tensor
$\mathfrak{g}$, we mean that $\mathfrak{g}$ is positive-definite, 
$\mathfrak{g}, \mathfrak{g}^{-1} \in C^{\infty}(\mathcal{M})$.

Recall that the Laplace-Beltrami operator $\Delta_\mathfrak{g}$, with 
$\mathfrak{g}=(\mathfrak{g}_{k\ell})\in \mathscr{G}$, is given in local 
coordinates by
\[
\Delta_\mathfrak{g}=\frac{1}{\sqrt{|\mathfrak{g}|}}
\sum_{k,\ell=1}^n\frac{\partial}{\partial x_k}\left(\sqrt{|\mathfrak{g}|}
\mathfrak{g}^{k\ell}\frac{\partial}{\partial x_\ell}\, \cdot \right),
\]
where $(\mathfrak{g}^{k\ell})$ is the inverse matrix to 
$(\mathfrak{g}_{k\ell})$ and $|\mathfrak{g}|$ is the determinant 
of $\mathfrak{g}$.

Let $dV_{\mathfrak{g}}=\sqrt{|\mathfrak{g}|}dx^1\ldots dx^n$ denote 
the Riemannian measure associated to $\mathfrak{g}$. 
We define an unbounded operator 
\[
A_{\mathfrak{g}}:L^2(\mathcal{M},dV_\mathfrak{g})\rightarrow 
L^2(\mathcal{M},dV_\mathfrak{g})
\] 
as follows:
\[
A_{\mathfrak{g}}=-\Delta _{\mathfrak{g}}\quad \mbox{with}\quad 
D\left(A_{\mathfrak{g}}\right)=H_0^1(\mathcal{M})\cap H^2(\mathcal{M}).
\]
Since $A_{\mathfrak{g}}$ is self-adjoint operator with compact resolvent,
its spectrum, denoted by $\sigma(A_{\mathfrak{g}})$, consists of a sequence 
of eigenvalues satisfying
\[
0 < \lambda_1^{\mathfrak{g}}<\lambda_2^{\mathfrak{g}}
\le  \ldots \lambda_k^{\mathfrak{g}}\le \ldots \quad \mbox{and}\quad 
\lambda_k^{\mathfrak{g}}\rightarrow \infty \; \mbox{as}\; k\rightarrow \infty.
\]

Furthermore, there exists an orthonormal basis $(\phi_k^ {\mathfrak{g}})$ 
of $L^2(\mathcal{M},dV_{\mathfrak{g}})$ consisting of eigenfunctions, where 
each $\phi_k^{\mathfrak{g}}$ is an eigenfunction for 
$\lambda_k^{\mathfrak{g}}$. By $\mathfrak{g} \in C^\infty(\mathcal{M})$ and the usual Sobolev 
regularity or H\"older regularity for elliptic equations, we can derive
$\phi_k^{\mathfrak{g}}\in C^\infty(\mathcal{M})$, $k\ge 1$. Henceforth we use the following notation: 
\[
\psi_k^{\mathfrak{g}}=\partial_\nu \phi_k^{\mathfrak{g}}|
{_{\partial \mathcal{M}}}, \quad k\ge 1,
\]
where $\nu$ is the unit normal exterior vector field on 
$\partial \mathcal{M}$ with respect to $\mathfrak{g}$.

For each $\mathfrak{g}\in \mathscr{G}$, we call the sequence $(\lambda_k^{\mathfrak{g}},\psi_k^{\mathfrak{g}})$ 
the boundary spectral data associated to $A_{\mathfrak{g}}$. We ask whether the boundary spectral data determine uniquely its corresponding 
Dirichlet-Laplace-Beltrami operator. In other words, does 
$(\lambda_k^{\mathfrak{g}},\psi_k^{\mathfrak{g}})$ determine uniquely 
$\mathfrak{g}$ ?

Before we give an answer to this question,  we need to introduce the notion 
of gauge equivalent operators. For $\chi\in C^\infty (\mathcal{M})$ satisfying $\chi >0$, we 
define an unbounded operator 
\[
A_\mathfrak{g}^\chi:L^2(\mathcal{M},\chi^{-2}dV)\rightarrow 
L^2(\mathcal{M},\chi^{-2}dV)
\]
by
\[
A_\mathfrak{g}^\chi=\chi A_\mathfrak{g}\chi^{-1},\quad D(A_\mathfrak{g}^\chi)
= D(A_\mathfrak{g}).
\]

Let $\mathfrak{g}_1, \mathfrak{g}_2\in \mathscr{G}$. We say that the operators 
$A_{\mathfrak{g}_1}$ and $A_{\mathfrak{g}_2}$ are gauge equivalent 
if $A_{\mathfrak{g}_2}=A_{\mathfrak{g}_1}^\chi$ for some 
$\chi \in C^\infty(\mathcal{M})$ satisfying $\chi >0$.

It is worth noticing that when $A_{\mathfrak{g}_1}$ and $A_{\mathfrak{g}_2}$ 
are gauge equivalent, 
we have $(\mathfrak{g}_1^{k\ell})=(\mathfrak{g}_2^{k\ell})$ 
(\cite[Section 2.2.9]{KKL}). 
Also, if $A_{\mathfrak{g}_1}$ and $A_{\mathfrak{g}_2}$ are gauge equivalent, 
then  $(\lambda^{\mathfrak{g}_2}_k)= (\lambda^{g_1}_k)$ (\cite[(2.54)]{KKL}) 
and $(\phi^{\mathfrak{g}_2}_k)= (\chi\phi^{\mathfrak{g}_1}_k)$ 
(\cite[(2.55)]{KKL}). 
Thus, if $\chi=1$ on $\partial \mathcal{M}$, then $A_{\mathfrak{g}_1}$ 
and $A_{\mathfrak{g}_2}$ have the same boundary spectral data. Conversely, if $A_{\mathfrak{g}_1}$ and $A_{\mathfrak{g}_2}$ have the same 
boundary spectral data, then $A_{\mathfrak{g}_1}$ and $A_{\mathfrak{g}_2}$ 
are gauge equivalent (\cite[Theorem 3.3]{KKL}).

Our analysis in the present section is based on the relationship 
between the boundary spectral data $(\lambda_k^{\mathfrak{g}},\psi_k^{\mathfrak{g}})$, $\mathfrak{g}\in \mathscr{G}$, and the so-called 
Dirichlet-to-Neumann map associated to $\mathfrak{g}$.

Before defining the Dirichlet-to-Neumann map, we observe that, 
since $0$ is not 
in the spectrum of $A_{\mathfrak{g}}$ with $\mathfrak{g}\in \mathscr{G}$, 
for every $f\in H^{\frac{1}{2}}(\partial \mathcal{M})$ the boundary value problem 
(BVP in short)
\begin{equation}\label{bvp}
\Delta _{\mathfrak{g}}u=0\; \mbox{in}\; \mathcal{M}\quad \mbox{and}\quad 
\mbox u=f\; \mbox{on}\; \partial \mathcal{M}
\end{equation}
admits a unique solution $u_{\mathfrak{g}}(f)\in H^1(\mathcal{M})$ 
with $\partial_\nu u_{\mathfrak{g}}(f)\in H^{-\frac{1}{2}}(\partial \mathcal{M})$. 
In addition, we have 
\begin{equation}\label{dn}
\|\partial_\nu u_{\mathfrak{g}}(f)\|_{H^{-\frac{1}{2}}(\partial \mathcal{M})}
\le C\|f\|_{H^{\frac{1}{2}}(\partial \mathcal{M})},
\end{equation}
where the constant $C>0$ is independent of $f$. 

We define the Dirichlet-to-Neumann map associated to 
$\mathfrak{g}\in \mathscr{G}$ as follows:
\[
\Lambda_{\mathfrak{g}} :f\in H^{\frac{1}{2}}(\partial \mathcal{M})\rightarrow 
\partial_\nu u_{\mathfrak{g}}(f)\in H^{-\frac{1}{2}}(\partial \mathcal{M}).
\]
It follows from \eqref{dn} that $ \Lambda_{\mathfrak{g}}\in 
\mathscr{B}(H^{\frac{1}{2}}(\partial \mathcal{M}),H^{-\frac{1}{2}}(\partial \mathcal{M}))$.

The determination of $\mathfrak{g}\in \mathscr{G}$ from the  
Dirichlet-to-Neumann map $\Lambda_{\mathfrak{g}}$ is  a great interest 
and appears in many applications. 
This inverse problem can be viewed for instance as the geometric 
formulation of the problem in which one wants to know whether it is possible 
to determine an anisotropic conductivity from the knowledge of boundary 
measurements. It is well known that there is an obstruction to uniqueness. 
Indeed, if $F:\mathcal{M}\rightarrow \mathcal{M}$ is 
a $C^\infty$-diffeomorphism which is the identity on the boundary, 
then $\Lambda_{F^\ast\mathfrak{g}}=\Lambda_{\mathfrak{g}}$, 
$\mathfrak{g}\in \mathscr{G}$, where $F^\ast\mathfrak{g}$ denotes 
the pull-back of $\mathfrak{g}$ by $F$.

 \subsection{Complete boundary spectral data}\label{subsection1.1}
 
Choose $\mathfrak{g}'\in \mathscr{G}$ and a constant $\mathbf{m} > 0$ 
satisfying $\mathbf{m}>\|\mathfrak{g}'\|_{W^{2,\infty}(\mathcal{M})}$. 
Let 
\[
\mathscr{G}'=\left\{\mathfrak{g}\in \mathscr{G};\; \mathfrak{g}\ge \mathfrak{g}',\; \mbox{and}\; \|\mathfrak{g}\|_{W^{2,\infty}(\mathcal{M})}\le \mathbf{m}\right\}.
\]
Also, fix $\mathfrak{g}_0\in \mathscr{G}'$ and define
\[
\mathscr{G}_0=\{\mathfrak{g}\in \mathscr{G}';\; \mathfrak{g}=\mathfrak{g}_0\; 
\mbox{on}\; \partial \mathcal{M}\}.
\]
 
Throughout this article, the sequences $(\alpha_k)$ and $(\beta_k)$ are 
respectively given by
\begin{equation}\label{se}
\alpha_k=k^{-\frac{1-2s}{2n}},\quad \beta_k=k^{\frac{1+2s}{n}},\quad  k\ge 1,
\end{equation}
where $0<s <1/2$ is fixed. For $\mathfrak{g}_1,\mathfrak{g}_2\in \mathscr{G}_0$, set  
\begin{equation}\label{dist}
\mathfrak{D}(\mathfrak{g}_1,\mathfrak{g}_2)
= \sum_{k\ge 1}\left[\alpha_k^2\left|\lambda_k^{\mathfrak{g}_1}
-\lambda_k^{\mathfrak{g}_2}\right|
+ \alpha_k\left\|\psi_k^{\mathfrak{g_1}}
-\psi_k^{\mathfrak{g}_2}\right\|_{L^2(\partial \mathcal{M})}\right] ,
\end{equation}
and
\begin{equation}\label{dist+}
\mathfrak{D}^+(\mathfrak{g}_1,\mathfrak{g}_2)
=\sum_{k\ge 1}\left[\beta_k\left|\lambda_k^{\mathfrak{g}_1}
-\lambda_k^{\mathfrak{g}_2}\right|+ \alpha_k\left\|\psi_k^{\mathfrak{g_1}}
-\psi_k^{\mathfrak{g}_2}\right\|_{L^2(\partial \mathcal{M})}\right].
\end{equation}

We say that $(\mathcal{M},\mathfrak{g})$, $\mathfrak{g}\in \mathscr{G}$, 
is admissible if $\mathcal{M}\Subset \mathbb{R}\times \tilde{\mathcal{M}}$
with some $(n-1)$-dimensional simple manifold 
$(\tilde{\mathcal{M}},\tilde{\mathfrak{g}})$ and 
$\mathfrak{g}=c(\mathfrak{e}\oplus \tilde{\mathfrak{g}})$, 
where $\mathfrak{e}$ is the Euclidean metric on $\mathbb{R}$ and 
$c\in C^\infty(\mathcal{M})$. 
Here by a simple manifold $(\tilde{\mathcal{M}},\tilde{\mathfrak{g}})$, we 
mean that for every $x\in \tilde{\mathcal{M}}$,
the exponential map $\exp_x$ with its maximal domain of definition is 
a diffeomorphism onto $\tilde{\mathcal{M}}$, 
and $\partial \tilde{\mathcal{M}}$ is strictly convex 
(that is, the second fundamental form on $\partial \tilde{\mathcal{M}}
\hookrightarrow \tilde{\mathcal{M}}$ is positive definite).
Examples of admissible manifolds can be found in \cite{DKSU}.

Define
\begin{equation}\label{Phi}
\Phi (\rho)=\ln\left(\rho+|\ln \rho|^{-1}\right),\quad 0<\rho<1.
\end{equation}
For fixed $\mathbf{r}>1$, we set 
\[
\mathscr{C}=\left\{c\in C^\infty (\mathcal{M});\; 
c=1\; \mbox{on}\; \partial \mathcal{M}\; \mbox{and}\; \|c\|_{C^3(\mathcal{M})}
+\|c^{-1}\|_{L^\infty (\mathcal{M})}\le \mathbf{r}\right\},
\]
and 
\[
\mathscr{G}_1=\{ c\mathfrak{g}_0;\; c\in \mathscr{C}\}.
\]
Note that $\mathscr{G}_1\subset \mathscr{G}_0$ for some $\mathbf{m}$ depending 
on $\mathbf{r}$ and $\mathfrak{g}_0$.

Our first main result is the following stability inequality.

\begin{theorem}\label{theoremCSD}
Assume that $(\mathcal{M}, \mathfrak{g}_0)$ is admissible. 
If $\mathfrak{g}_1=c_1\mathfrak{g}_0\in \mathscr{G}_1$ 
and $\mathfrak{g}_2=c_2\mathfrak{g}_0\in \mathscr{G}_1$ satisfy 
$\mathfrak{D}^+(\mathfrak{g}_1,\mathfrak{g}_2)<\infty$ and 
$\mathfrak{D}=\mathfrak{D}(\mathfrak{g}_1,\mathfrak{g}_2)\le e^{-\varrho}$, 
then we have
\[
\|c_1-c_2\|_{L^\infty(\mathcal{M})}
\le C\left[\Phi \left(\mathfrak{D}\right)\right]^{-\theta},
\]
where $C>0$ and $\varrho>0$ depend only on $n$, 
$\mathcal{M}$, $\mathfrak{g}_0$, $s$ and $\mathbf{r}$, while $\theta$ 
depends only on $n$. 
\end{theorem}
 
\subsection{Partial boundary spectral data on a set of positive measure}\label{subsection1.2}
 
In this subsection, we suppose that $\mathcal{M}=\overline{\Omega}$, 
where $\Omega$ is $C^\infty$ bounded domain of $\mathbb{R}^n$. Furthermore, assume that $\Omega$ is chosen so that 
$\partial \Omega$ has a $C^\infty$ connected neighborhood $N$ 
in $\overline{\Omega}$.

Henceforth let $\Sigma$ be a measurable subset of $\partial \Omega$ satisfying
\[
|\Sigma \cap B(\tilde{x},r)|>0,\quad 0<r\le r_0,
\]
for some  $\tilde{x}\in \partial \Omega$ and  $r_0>0$,
where  $B(\tilde x, r):= \{ x\in \Omega;\, 
\vert x-\tilde{x} \vert < r\}$.

Consider the notations 
\begin{align*}
&\mathbf{D}=\mathbf{D}(\mathfrak{g}_1,\mathfrak{g}_2)
= \sum_{k\ge 1}\alpha_k\left\|\psi_k^{\mathfrak{g}_1}
-\psi_k^{\mathfrak{g}_2}\right\|_{L^\infty(\Sigma)}^{\frac{\alpha(1-2s)}{4}},
\\
&\mathbf{D}_0=\mathbf{D}_0(\mathfrak{g}_1,\mathfrak{g}_2)
=\sum_{k\ge 1}\left\|\psi_k^{\mathfrak{g}_1}-\psi_k^{\mathfrak{g}_1}\right\|
_{L^2(\partial \Omega)}.
\end{align*}
Hereinafter, 
\[
\displaystyle \Psi(\rho)= 
\left\{
\begin{array}{ll}
\rho, &0\le \rho\le 1,
\\
 (\ln \rho)^{-\frac{1-2s}{4}},\quad &\rho>1 ,
\end{array}
\right.
\]
and
\[
\Psi_\tau=\tau \Psi,\quad  \Phi_\tau=\Phi\circ \Psi_\tau,\quad \tau>0,
\]
where $\Phi$ is given by \eqref{Phi}.

\begin{theorem}\label{theoremPSD}
Assume that $(\overline{\Omega},\mathfrak{g}_0)$ is admissible. 
Let $\mathfrak{g}_1=c_1\mathfrak{g}_0\in \mathscr{G}_1$ and 
$\mathfrak{g}_2=c_2\mathfrak{g}_0\in \mathscr{G}_1$ satisfy $c_1=c_2$ in $N$, 
$\sigma (A_{\mathfrak{g}_1})=\sigma (A_{\mathfrak{g}_2})$, 
$\mathbf{D}=\mathbf{D}(\mathfrak{g}_1,\mathfrak{g}_2)<\infty$, 
$\mathbf{D}_0=\mathbf{D}_0(\mathfrak{g}_1,\mathfrak{g}_2)< \infty$ 
and $\mathfrak{D}(\mathfrak{g}_1,\mathfrak{g}_2)\le e^{-\varrho}$.  
Then
\[
\|c_1-c_2\|_{L^\infty(\mathcal{M})}\le C\left[\Phi_\tau 
\left(\frac{\mathbf{D}_0}{\mathbf{D}}\right)\right]^{-\theta},
\]
where $C>0$ and $c>0$ depend only on $n$, $\Omega$, $\Sigma$, 
$\mathfrak{g}_0$, $s$ and $\mathbf{r}$, while $\varrho$ and  $\theta$ are 
as in Theorem \ref{theoremCSD}.
\end{theorem}

We modify slightly the proof of Theorem \ref{theoremPSD} in order to derive 
a uniqueness result for the problem of determining the metric tensor 
$\mathfrak{g}$ from the partial Dirichlet-to-Neumann map 
$\Lambda_{\mathfrak{g}}{_{|\Sigma}}$. Set
\[
\tilde{\mathscr{C}}
= \left\{c\in C^\infty (\overline{\Omega});\;  c=1\; \mathrm{on}\; 
\partial \Omega\right\}.
\]

\begin{theorem}\label{theoremU}
Suppose that $(\overline{\Omega},\mathfrak{g}_0)$ is admissible. 
Let $\mathfrak{g}_j=c_j\mathfrak{g}_0$ with $c_j\in \tilde{\mathscr{C}}$, 
$j=1,2$. If $c_1=c_2$ in $N$ and $\Lambda_{\mathfrak{g}_1}(f){_{|\Sigma}}
=\Lambda_{\mathfrak{g}_2}(f){_{|\Sigma}}$.
Then $c_1=c_2$.
\end{theorem}

\subsection{Partial boundary spectral data on an open subset}
\label{subsection1.3}

Theorem \ref{theoremPSD} can be improved when $\Sigma$ is substituted 
by a nonempty connected open subset $\Gamma$ of $\partial \mathcal{M}$. 
In this case the condition $\sigma (A_{\mathfrak{g}_1})
= \sigma (A_{\mathfrak{g}_2})$ in Theorem \ref{theoremPSD} becomes unnecessary.
 
Fix $\delta >0$ and consider the following new notations:
 \begin{align*}
 &\mathscr{D}(\mathfrak{g}_1,\mathfrak{g}_2)
= \sum_{k\ge 1}\left[ \left|\lambda_k^{\mathfrak{g}_2}
-\lambda_k^{\mathfrak{g}_1}\right|+\left\|\psi_k^{\mathfrak{g}_1}
-\psi_k^{\mathfrak{g}_2}\right\|_{L^2(\partial \mathcal{M})}\right],
\\
 &\mathscr{D}_0(\mathfrak{g}_1,\mathfrak{g}_2)
= \sum_{k\ge 1}\left[ \left|\lambda_k^{\mathfrak{g}_2}
-\lambda_k^{\mathfrak{g}_1}\right|+\left\|\psi_k^{\mathfrak{g}_1}
-\psi_k^{\mathfrak{g}_2}\right\|_{L^2(\Gamma)}\right],
\\
&\tilde{\mathscr{D}}_0(\mathfrak{g}_1,\mathfrak{g}_2)
= \sum_{k\ge 1}k^{\delta}\left\|\psi_k^{\mathfrak{g}_1}
-\psi_k^{\mathfrak{g}_2}\right\|_{L^2(\partial \mathcal{M})}.
\end{align*}
 
Let $\mathcal{N}$ be a neighborhood of $\partial \mathcal{M}$ in $\mathcal{M}$.
We assume that $\mathcal{N}$ is also $n$-dimensional connected compact 
manifold.
 
\begin{theorem}\label{theoremPSD2}
Assume that $(\mathcal{M}, \mathfrak{g}_0)$ is admissible. 
Let $\mathbf{C}_0>0$, $\mathfrak{g}_j=c_j\mathfrak{g}_0\in \mathscr{G}_1$, 
$j=1,2$ such that $c_1=c_2$ in $\mathcal{N}$, 
$\mathscr{D}_0=\mathscr{D}_0(\mathfrak{g}_1,\mathfrak{g}_2)<\infty$,  
$\tilde{\mathscr{D}}_0(\mathfrak{g}_1,\mathfrak{g}_2)\le \mathbf{C}_0$ and 
$\mathfrak{D}(\mathfrak{g}_1,\mathfrak{g}_2)\le e^{-\varrho}$. 
Then $\mathscr{D}(\mathfrak{g}_1,\mathfrak{g}_2)<\infty$ and
\[
\|c_1-c_2\|_{L^\infty(\mathcal{M})}\le C\left[\Phi_\tau\left(\mathscr{D}_0
\right)\right]^{-\theta},
\]
where the constant $C>0$ depends only on $n$, $\mathcal{M}$, $s$, 
$\mathfrak{g}_0$, $\mathbf{r}$, $\mathcal{N}$ and $\Gamma$, 
the constant $\tau$ depends only on 
$n$, $\mathcal{M}$, $s$, $\mathfrak{g}_0$, $\mathbf{r}$, $\mathbf{C}_0$, 
$\mathcal{N}$ and $\Gamma$, while $\varrho$ and  $\theta$ are 
as in Theorem \ref{theoremCSD}.
\end{theorem}
 
\subsection{Comments}
 
Theorem \ref{theoremCSD} corresponds to a stability inequality for 
the uniqueness result \cite[Theorem 3.3]{KKL}, while Theorem \ref{theoremPSD} 
gives a stability inequality for the uniqueness result in \cite{Ch22}. To the best of our knowledge,
there are no other stability estimates in the literature concerning 
the determination of a Dirichlet-Laplace-Beltrami operator from its boundary 
spectral data.

The first uniqueness result for the mutidimensional Borg-Levinson theorem was 
obtained in \cite{NSU} and improved later in \cite{Is}. 
A generalization of the result in \cite{Is} was established in \cite{ChS}. 
The case of bounded potentials in Schr\"odinger equations 
in a periodic waveguide was considered in \cite{KKS}, 
where new ideas were introduced. 
The case of unbounded potentials was first studied by \cite{PS} and 
continued in \cite{BKMS,KS,Po}. While the  determination of the magnetic 
field in a magnetic Sch\"odinger equation from its boundary spectral data was 
treated in \cite{Ki1}.

A first stability inequality for the mutidimensional Borg-Levinson theorem was 
established in \cite{AS} and improved in \cite{ChS}. The case of partial 
boundary spectral data was studied in \cite{BCY,BCY2}. Stability inequalities 
for both the potential and the magnetic field were recently established in 
\cite{BCDKS} for a magnetic Sch\"odinger operator on a connected compact 
Riemannian manifold.
We can refer to \cite{Bel87,Bel92,BD,Ch09,CMS,KK, KOM,Po, So}, but here we do 
not intend any comprehensive references.

The two dimensional case was recently studied in \cite{IY} 
where the determination of the metric tensor from partial boundary spectral 
data on an arbitrary subboundary was considered. 
A logarithmic stability inequality was established for this problem.

\section{Proof of the theorems of Section \ref{section1}}

\subsection{Proof of Theorem \ref{theoremCSD}}

The scalar products of $L^2(\mathcal{M},dV_\mathfrak{g})$ 
and $L^2(\partial \mathcal{M},dS_\mathfrak{g})$ will be denoted 
hereinafter respectively by $(\cdot |\cdot )_\mathfrak{g}$ and 
$\langle \cdot |\cdot \rangle_\mathfrak{g}$. For $\lambda\ge 0$ and $f\in H^{\frac{1}{2}}(\partial \mathcal{M})$, 
we denote by $u_{\mathfrak{g}}^\lambda (f)\in H^1(\mathcal{M})$ 
the solution of the BVP
\[
(-\Delta _{\mathfrak{g}}+\lambda) u=0\; \mbox{in}\; \mathcal{M}\quad \mbox{and}\quad u=f\; \mbox{on}\; \partial \mathcal{M}.
\]

\begin{lemma}\label{lemma0}
Let $\lambda\ge 0$ and $f\in H^{\frac{1}{2}}(\partial \mathcal{M})$. Then we have
\begin{equation}\label{eq0}
u_{\mathfrak{g}}^\lambda(f)=-\sum_{k\ge 1}\frac{\langle f|\psi_k^{\mathfrak{g}}\rangle_\mathfrak{g}}{\lambda_k^{\mathfrak{g}}+\lambda}\phi_k^{\mathfrak{g}}
\end{equation}
and hence
\[
\sum_{k\ge 1}\frac{|\langle f|\psi_k^{\mathfrak{g}}\rangle_\mathfrak{g}|^2}{(\lambda_k^{\mathfrak{g}}+\lambda)^2}=\|u_{\mathfrak{g}}^\lambda(f)\|_{L^2(\mathcal{M})}^2.
\]
\end{lemma}

\begin{proof}
Using the identity
\begin{equation}\label{eq2}
u_{\mathfrak{g}}^\lambda(f)=\sum_{k\ge 1}(u_{\mathfrak{g}}^\lambda(f)|\phi_k^{\mathfrak{g}})_\mathfrak{g}\phi_k^{\mathfrak{g}}
\end{equation}
and Green's formula we get
\begin{align*}
(u_{\mathfrak{g}}^\lambda(f)|(\lambda_k^{\mathfrak{g}}+\lambda)\phi_k^{\mathfrak{g}})_\mathfrak{g}&=(u_{\mathfrak{g}}^\lambda(f)| -\Delta_{\mathfrak{g}}\phi_k^{\mathfrak{g}})_\mathfrak{g}+(u_{\mathfrak{g}}^\lambda(f)|\lambda\phi_k^{\mathfrak{g}})_\mathfrak{g}
\\
&= ((-\Delta_{\mathfrak{g}}+\lambda)u_{\mathfrak{g}}^\lambda(f)|\phi_k^{\mathfrak{g}})_\mathfrak{g}-\langle u_{\mathfrak{g}}^\lambda(f)|\partial_\nu \phi_k^{\mathfrak{g}}\rangle_\mathfrak{g}
\\
&=-\langle f|\partial_\nu \phi_k^{\mathfrak{g}}\rangle_\mathfrak{g} .
\end{align*}
That is we have
\[
(u_{\mathfrak{g}}^\lambda(f)|\phi_k^{\mathfrak{g}})_\mathfrak{g}=-\frac{\langle f|\psi_k^{\mathfrak{g}}\rangle_\mathfrak{g}}{\lambda_k^{\mathfrak{g}}+\lambda}.
\]
Hence \eqref{eq0} follows from \eqref{eq2}.
\end{proof}

\begin{lemma}\label{lemma1}
We have
\begin{equation}\label{e3}
\left\|\psi_k^{\mathfrak{g}}\right\|_{L^2(\partial \mathcal{M},dS_\mathfrak{g})}\le Ck^{\frac{3+2s}{2n}},\quad k\ge 1,\; \mathfrak{g}\in \mathscr{G}',
\end{equation}
where the constant $C>0$ only depends on $n$, $\mathcal{M}$, $\mathfrak{g}'$, $s$ and $\mathbf{m}$.
\end{lemma}

\begin{proof}
In this proof $C>0$ denotes a generic constant depending only on $n$, $\mathcal{M}$, $s$, $\mathfrak{g}'$ and $\mathbf{m}$.
Let $\mathfrak{g}\in \mathscr{G}'$. It follows from the usual a priori estimate in $H^2$ that
\[
\|\phi^{\mathfrak{g}}\|_{H^2(\mathcal{M})}\le C\lambda_k^{\mathfrak{g}},\quad k\ge 1,
\]
which, combined with the interpolation inequality 
\[
\|w\|_{H^{\frac{3}{2}+s}(\mathcal{M})}\le C\|w\|_{H^2(\mathcal{M})}^{\frac{3+2s}{4}}\|w\|_{L^2(\mathcal{M})}^{\frac{1-2s}{4}},\quad w\in H^2(\mathcal{M})
\]
(e.g. \cite[Subsection 2.1 in page 40]{LM}), gives
\begin{equation}\label{e1}
\left\|\phi_k^{\mathfrak{g}}\right\|_{H^{\frac{3}{2}+s}(\mathcal{M})}\le C\left(\lambda_k^{\mathfrak{g}}\right)^{\frac{3+2s}{4}},\quad k\ge 1.
\end{equation}
From Weyl's asymptotic formula and the min-max principle we have
\begin{equation}\label{wf}
\varkappa^{-1}k^{\frac{2}{n}}\le \lambda_k^{\mathfrak{g}}\le \varkappa k^{\frac{2}{n}},
\end{equation}
where $\varkappa>1$ depends only on $n$, $\mathcal{M}$, $\mathfrak{g}'$ and $\mathbf{m}$. Then \eqref{e1} implies
\begin{equation}\label{e2}
\left\|\phi_k^{\mathfrak{g}}\right\|_{H^{\frac{3}{2}+s}(\mathcal{M})}\le Ck^{\frac{3+2s}{2n}},\quad k\ge 1.
\end{equation}
Using that the trace operator 
\begin{equation}\label{to}
\mathfrak{t}:w\in H^{\frac{3}{2}+s}(\mathcal{M})\mapsto \partial_\nu w\in L^2(\partial \mathcal{M})
\end{equation}
 is bounded (e.g. \cite[Theorem 9.4 in page 41]{LM}) we derive from \eqref{e2} 
\[
\left\|\psi_k^{\mathfrak{g}}\right\|_{L^2(\partial \mathcal{M},dS_\mathfrak{g})}\le Ck^{\frac{3+2s}{2n}},\quad k\ge 1,\; \mathfrak{g}\in \mathscr{G}'.
\]
This is the expected inequality.
\end{proof}

\begin{theorem}\label{theoremDN}
Let $\mathfrak{g}_1,\mathfrak{g}_2\in \mathscr{G}_0$ satisfying $\mathfrak{D}^+(\mathfrak{g}_1,\mathfrak{g}_2)<\infty$. Then $\Lambda_{\mathfrak{g}_1}-\Lambda_{\mathfrak{g}_2}$ extends to a bounded operator on $L^2(\partial \mathcal{M})$ and
\begin{equation}\label{a5}
\|\Lambda_{\mathfrak{g}_1}-\Lambda_{\mathfrak{g}_2}\|_{\mathscr{B}(L^2(\partial \mathcal{M}))}\le C\mathfrak{D}(\mathfrak{g}_1,\mathfrak{g}_2),
\end{equation}
where the constant $C>0$ depends only on $n$, $\mathcal{M}$, $s$, $\mathfrak{g}'$, $\mathfrak{g}_0$ and $\mathbf{m}$.
\end{theorem}

\begin{proof}
In this proof $C>0$ denotes a generic constant only depending on $n$, $\mathcal{M}$, $s$, $\mathfrak{g}'$, $\mathfrak{g}_0$ and $\mathbf{m}$.
Since $dS_{\mathfrak{g}}=dS_{\mathfrak{g}_0}$ for each $\mathfrak{g}\in \mathscr{G}_0$, we use henceforth the notation $L^2(\partial \mathcal{M})$ (resp. $\langle \cdot |\cdot \rangle$) instead of $L^2(\partial \mathcal{M},dS_{\mathfrak{g}_0})$  (resp. $\langle \cdot |\cdot \rangle_{\mathfrak{g}_0}$). 
 
 Pick $\mathfrak{g}_1,\mathfrak{g}_2\in \mathscr{G}_0$ so that $\mathfrak{D}^+(\mathfrak{g}_1,\mathfrak{g}_2)<\infty$. Let $f\in H^{\frac{3}{2}}(\partial \mathcal{M})$  and set
\[
a_k^{\mathfrak{g}_1}=-\frac{\langle f|\psi_k^{\mathfrak{g}_1}\rangle}{\lambda_k^{\mathfrak{g}_1}}\psi_k^{\mathfrak{g}_1},\quad a_k^{\mathfrak{g}_2}=-\frac{\langle f|\psi_k^{\mathfrak{g}_2}\rangle}{\lambda_k^{\mathfrak{g}_2}}\psi_k^{\mathfrak{g}_2}.
\]
We split
\[
w_\ell=-\sum_{k\le \ell}\left(a_k^{\mathfrak{g}_1}- a_k^{\mathfrak{g}_2}\right),\quad \ell \ge 1,
\]
 into three terms: $w_\ell=w_\ell^1+w_\ell^2+w_\ell ^3$ with
\begin{align*}
&w_\ell^1= \sum_{k\le \ell} \left(\frac{1}{\lambda_k^{\mathfrak{g}_2}}-\frac{1}{\lambda_k^{\mathfrak{g}_1}}\right)\langle f|\psi_k^{\mathfrak{g}_2}\rangle\psi_k^{\mathfrak{g}_2},
\\
&w_\ell^2=\sum_{k\le \ell} \frac{\langle f|\psi^{\mathfrak{g}_2}_k\rangle-\langle f|\psi_k^{\mathfrak{g}_1}\rangle}{\lambda_k^{\mathfrak{g}_1}}\psi_k^{\mathfrak{g}_2},
\\
&w_\ell^3=\sum_{k\le \ell} \frac{\langle f|\psi_k^{\mathfrak{g}_1}\rangle}{\lambda_k^{\mathfrak{g}_1}}\left(\psi_k^{\mathfrak{g}_2}-\psi_k^{\mathfrak{g}_1}\right).
\end{align*}
Taking into account that $\mathfrak{D}(\mathfrak{g_1},\mathfrak{g}_2)\le \mathfrak{D}^+(\mathfrak{g_1},\mathfrak{g}_2)<\infty$,  we obtain  from \eqref{e3} and \eqref{wf}
\begin{align*}
&\left\|w_\ell^1\right\|_{L^2\left(\partial \mathcal{M}\right)}\le C\sum_{k\le \ell} \alpha_k^2\left|\lambda_k^{\mathfrak{g}_1}-\lambda_k^{\mathfrak{g}_2}\right|\le C\mathfrak{D}(\mathfrak{g_1},\mathfrak{g}_2) ,
\\
&\left\|w_\ell^2\right\|_{L^2\left(\partial \mathcal{M}\right)}+\left\|w_\ell^3\right\|_{L^2\left(\partial \mathcal{M}\right)} \le C\sum_{k\le \ell} \alpha_k \left\|\psi_k^{\mathfrak{g}_1}-\psi_k^{\mathfrak{g}_2}\right\|_{L^2(\partial \mathcal{M})}\le C\mathfrak{D}(\mathfrak{g_1},\mathfrak{g}_2).
\end{align*}
Hence the series  $\sum_{k\ge 1}\left(a_k^{\mathfrak{g}_1}-a_k^{\mathfrak{g}_2}\right)$ converges  in $L^2(\partial \mathcal{M})$. Furthermore, the following inequality holds
\begin{equation}\label{e4}
\left\| \sum_{k\ge 1}\left(a_k^{\mathfrak{g}_1}-a_k^{\mathfrak{g}_2}\right) \right\|_{L^2(\partial \mathcal{M})}\le C\mathfrak{D}(\mathfrak{g_1},\mathfrak{g}_2).
\end{equation}
For an arbitrary $g\in \mathscr{G}_0$ we have
\begin{equation}\label{a1}
u_{\mathfrak{g}}(f)-u_{\mathfrak{g}}^\lambda(f)=\sum_{k\ge 1}b_k\phi_k^{\mathfrak{g}},
\end{equation}
where
\[
b_k=-\frac{\lambda\langle f|\psi_k^{\mathfrak{g}}\rangle}{\lambda_k^{\mathfrak{g}}(\lambda_k^{\mathfrak{g}}+\lambda)}.
\]
Whence
\[
\left(\lambda_k^{\mathfrak{g}}\right)^2|b_k|^2=\frac{\lambda^2 \langle f|\psi_k^{\mathfrak{g}}\rangle|^2}{(\lambda_k^{\mathfrak{g}}+\lambda)^2}.
\]
In light of Lemma \ref{lemma0}, $\displaystyle \sum_{k\geq 1}\frac{\lambda^2|\langle f|\psi_k^{\mathfrak{g}}\rangle|^2}{(\lambda_k^{\mathfrak{g}}+\lambda)^2}<\infty$ and then the series in \eqref{a1} converges in $H^2(\mathcal{M})$.  Hence
\begin{equation}\label{a2}
\mathfrak{t} u_{\mathfrak{g}}(f)-\mathfrak{t} u_{\mathfrak{g}}^\lambda(f)=\sum_{k\ge 1}b_k\psi_k^{\mathfrak{g}}=\sum_{k\ge 1}\frac{\lambda}{\lambda_k^{\mathfrak{g}}+\lambda}a_k^\mathfrak{g}
\end{equation}
(the operator $\mathfrak{t}$ is defined in \eqref{to}) and the series in the right hand side converges in $L^2(\partial \mathcal{M})$.

The preceding calculations give
\begin{align}
\mathfrak{t} u_{\mathfrak{g}_1}(f)-\mathfrak{t} u_{\mathfrak{g}_2}(f)-&\sum_{k\ge 1}(a_k^{\mathfrak{g}_1}-a_k^{\mathfrak{g}_2})=\mathfrak{t} u_{\mathfrak{g}_1}^\lambda(f)-\mathfrak{t} u_{\mathfrak{g}_2}^\lambda (f)\label{a3}
\\
&+\sum_{k\ge 1}\left[\left(\frac{\lambda }{\lambda+\lambda_k^{\mathfrak{g}_1}}-1\right)a_k^{\mathfrak{g}_1}-\left(\frac{\lambda}{\lambda+\lambda_k^{\mathfrak{g}_2}}-1\right)a_k^{\mathfrak{g}_2}\right].\nonumber
\end{align}
Define
\[
a_k^{\mathfrak{g}_1}(\lambda)=-\frac{\langle f|\psi_k^{\mathfrak{g}}\rangle}{\lambda_k^{\mathfrak{g}_1}+\lambda}\psi_k^{\mathfrak{g}_1},\quad a_k^{\mathfrak{g}_2}(\lambda)=-\frac{\langle f|\psi_k^{\mathfrak{g}_2}\rangle}{\lambda_k^{\mathfrak{g}_2}+\lambda}\psi_k^{\mathfrak{g}_2},\quad \lambda \ge 0.
\]
As  for the case $\lambda=0$ we split
\[
w_\ell(\lambda)=-\sum_{k\le \ell}\left[a_k^{\mathfrak{g}_1}(\lambda)- a_k^{\mathfrak{g}_2}(\lambda)\right],\quad \ell \ge 1,
\]
 into three terms: $w_\ell(\lambda)=w_\ell^1(\lambda)+w_\ell^2(\lambda)+w_\ell ^3(\lambda)$, with
\begin{align*}
&w_\ell^1(\lambda)= \sum_{k\le \ell} \left(\frac{1}{\lambda_k^{\mathfrak{g}_2}+\lambda}-\frac{1}{\lambda_k^{\mathfrak{g}_1}+\lambda}\right)\langle f|\psi_k^{\mathfrak{g}_2}\rangle\psi_k^{\mathfrak{g}_2},
\\
&w_\ell^2(\lambda)=\sum_{k\le \ell} \frac{\langle f|\psi^{\mathfrak{g}_2}_k\rangle-\langle f|\psi_k^{\mathfrak{g}_1}\rangle}{\lambda_k^{\mathfrak{g}_1}+\lambda}\psi_k^{\mathfrak{g}_2},
\\
&w_\ell^3(\lambda)=\sum_{k\le \ell} \frac{\langle f|\psi_k^{\mathfrak{g}_1}\rangle}{\lambda_k^{\mathfrak{g}_1}+\lambda}\left(\psi_k^{\mathfrak{g}_2}-\phi_k^{\mathfrak{g}_1}\right).
\end{align*}
Taking into account the calculations we carried out above, we derive with the help of dominated convergence theorem that
\[
\lim_{\lambda \rightarrow \infty} \left\| \sum_{k\ge 1}\left[a_k^{\mathfrak{g}_1}(\lambda)-a_k^{\mathfrak{g}_2}(\lambda)\right]\right\|_{L^2(\partial \mathcal{M})}=0.
\]
That is we have
\begin{equation}\label{a4}
\lim_{\lambda \rightarrow \infty}\left\| \mathfrak{t} u_{\mathfrak{g}_1}^\lambda(f)-\mathfrak{t} u_{\mathfrak{g}_2}^\lambda (f)\right\|_{L^2(\partial \mathcal{M})}=0.
\end{equation}
Since
\[
\left\| \left(\frac{\lambda }{\lambda+\lambda_k^{\mathfrak{g}_1}}-1\right)(a_k^{\mathfrak{g}_1}-a_k^{\mathfrak{g}_2})\right\|_{L^2(\partial \mathcal{M})}\le 2 \left\|a_k^{\mathfrak{g}_1}-a_k^{\mathfrak{g}_2}\right\|_{L^2(\partial \mathcal{M})}
\]
we can apply again dominated convergence theorem  to derive 
\begin{equation}\label{e6}
\lim_{\lambda \rightarrow \infty}\left\|\sum_{k\ge 1}\left(\frac{\lambda }{\lambda+\lambda_k^{\mathfrak{g}_1}}-1\right)(a_k^{\mathfrak{g}_1}-a_k^{\mathfrak{g}_2})\right\|_{L^2(\partial \mathcal{M})}=0.
\end{equation}
On the other hand, the inequality  
\[
\left\|a_k^{\mathfrak{g}_2}\right\|_{L^2(\partial \mathcal{M})}\le C\beta_k\|f\|_{L^2(\partial \mathcal{M})}
\] 
yields  
\[
\left\|\left(\frac{\lambda }{\lambda+\lambda_k^{\mathfrak{g}_1}}-\frac{\lambda}{\lambda+\lambda_k^{\mathfrak{g}_2}}\right)a_k^{\mathfrak{g}_2}\right\|_{L^2(\partial \mathcal{M})}\le C\lambda ^{-1}\|f\|_{L^2(\partial \mathcal{M})}\beta_k\left|\lambda_k^{\mathfrak{g}_1}-\lambda_k^{\mathfrak{g}_2}\right|.
\]
This inequality together with $\mathfrak{D}^+(\mathfrak{g}_1,\mathfrak{g}_2)<\infty$ imply
\begin{equation}\label{e7}
\lim_{\lambda \rightarrow \infty}\left\|\sum_{k\ge 1}\left(\frac{\lambda }{\lambda+\lambda_k^{\mathfrak{g}_1}}-\frac{\lambda}{\lambda+\lambda_k^{\mathfrak{g}_2}}\right)a_k^{\mathfrak{g}_2}\right\|_{L^2(\partial \mathcal{M})}=0.
\end{equation}
Putting together \eqref{e6} and \eqref{e7} we obtain
\begin{equation}\label{e5}
\lim_{\lambda \rightarrow \infty}\left\| \sum_{k\ge 1}\left[\left(\frac{\lambda }{\lambda+\lambda_k^{\mathfrak{g}_1}}-1\right)a_k^{\mathfrak{g}_1}-\left(\frac{\lambda}{\lambda+\lambda_k^{\mathfrak{g}_2}}-1\right)a_k^{\mathfrak{g}_2}\right]\right\|_{L^2(\partial \mathcal{M})}=0.
\end{equation}
Finally, combining \eqref{a3}, \eqref{a4} and  \eqref{e4} we find
\[
\|(\Lambda_{\mathfrak{g}_1}-\Lambda_{\mathfrak{g}_2})(f)\|_{L^2(\partial \mathcal{M})}=\left\|\sum_{k\ge 1}(a_k^{\mathfrak{g}_1}-a_k^{\mathfrak{g}_2})\right\|_{L^2(\partial \mathcal{M})}\le C\mathfrak{D}(\mathfrak{g}_1,\mathfrak{g}_2)\|f\|_{L^2(\partial \mathcal{M})}.
\]
In other words, $\Lambda_{\mathfrak{g}_1}-\Lambda_{\mathfrak{g}_2}$ extends to a bounded operator on $L^2(\partial \mathcal{M})$ and \eqref{a5} holds.
\end{proof}

The proof of Theorem \ref{theoremDN} was inspired by that of  \cite[Theorem 1]{Ch21}.

Next, using the following inequality
\[
\|\Lambda_{\mathfrak{g}_1}-\Lambda_{\mathfrak{g}_2}\|_{\mathscr{B}(H^{\frac{1}{2}}(\partial \mathcal{M}),H^{-\frac{1}{2}}(\partial \mathcal{M}))}\le \|\Lambda_{\mathfrak{g}_1}-\Lambda_{\mathfrak{g}_2}\|_{\mathscr{B}(L^2(\partial \mathcal{M}))}
\]
and \eqref{a5} we get
\begin{equation}\label{a5.1}
\|\Lambda_{\mathfrak{g}_1}-\Lambda_{\mathfrak{g}_2}\|_{\mathscr{B}(H^{\frac{1}{2}}(\partial \mathcal{M}),H^{-\frac{1}{2}}(\partial \mathcal{M}))}\le C\mathfrak{D}(\mathfrak{g}_1,\mathfrak{g}_2),
\end{equation}
where $C>0$ only depends on $n$, $\mathcal{M}$, $s$, $\mathfrak{g}'$, $\mathfrak{g}_0$ and $\mathbf{m}$.

In light of \eqref{a5.1} Theorem \ref{theoremCSD} is a  direct consequence of \cite[Theorem 2]{CS}. For reader convenience we give the statement of slightly different version of \cite[Theorem 2]{CS}. The norm of $\mathscr{B}(H^{\frac{1}{2}}(\partial \mathcal{M}),H^{-\frac{1}{2}}(\partial \mathcal{M}))$ will be denoted by $\|\cdot \|_{\frac{1}{2}}$.

\begin{theorem}\label{theoremCS}
Assume that $\mathcal{M}$ is admissible. Let $\mathfrak{g}_j=c_j\mathfrak{g}_0\in \mathscr{G}_1$, $j=1,2$, satisfy $\|\Lambda_{\mathfrak{g}_1}-\Lambda_{\mathfrak{g}_2}\|_{1/2}\le e^{-\varrho}$. Then we have
\[
\|c_1-c_2\|_{L^\infty(\mathcal{M})}\le C\left[\Phi \left(\|\Lambda_{\mathfrak{g}_1}-\Lambda_{\mathfrak{g}_2}\|_{\frac{1}{2}}\right)\right]^{-\theta},
\]
where $C>0$ and $\varrho>0$ depend only on $n$, $\mathcal{M}$, $\mathfrak{g}_0$  and $\mathbf{c}$, while $\theta$ only depends on $n$. 
\end{theorem}

We mention that a stability inequality for the problem of determining the metric tensor at the boundary from a local Dirichlet-to-Neumann map  was established in \cite{KY}.

\subsection{Proof of Theorem \ref{theoremPSD}}

Pick $g\in \mathscr{G}'$ and $\lambda \ge 0$ and let $u\in W^{2,\infty}(N)$ be a solution of the following BVP
\begin{equation}\label{bvp1}
(\Delta_{\mathfrak{g}}+\lambda )u=0\; \mathrm{in}\; N\quad u=0\; \mathrm{on}\; \Gamma.
\end{equation}
From \cite[Theorem 1.1]{Ch22} it holds 
\begin{equation}\label{a7}
\|u\|_{L^2(N)}  \le Ce^{\lambda r_1}\|\partial_\nu u\|_{L^\infty(\Sigma)}^\alpha \|u\|_{W^{1,\infty}(N)}^{1-\alpha},
\end{equation}
where $C>0$, $0<r_1\le r_0$ and $0<\alpha<1$ only depend on $n$, $\Sigma$, $N$, $\mathfrak{g}'$ and $\mathbf{m}$.

From now on the notations and the assumptions are those of Subsection \ref{subsection1.2}.
$C>0$, $c>0$ and $c'>0$ will denote generic constants only depending on $n$, $\Sigma$, $N$, $\tilde{x}$, $r_0$, $\mathfrak{g}'$, $s$ and $\mathbf{m}$. The constants $r_1$ and $\alpha$ are as in \eqref{a7}.

Let $\mathfrak{g}_j=c_j\mathfrak{g}_0\in \mathscr{G}_1$, $j=1,2$, so that $c_1=c_2$ in $N$ and $\sigma(A_{\mathfrak{g}_1})=\sigma(A_{\mathfrak{g}_2})$. Then $u_k=\phi_k^{\mathfrak{g}_1}- \phi_k^{\mathfrak{g}_2}$, $k\ge 1$, is a solution of \eqref{bvp1} with $\lambda=\lambda_k^{\mathfrak{g}_1}=\lambda_k^{\mathfrak{g}_2}$. Applying \eqref{a7} with $u=u_k$, we get
\begin{equation}\label{a8}
\left\|\phi_k^{\mathfrak{g}_1}- \phi_k^{\mathfrak{g}_2}\right\|_{L^2(N)}  \le Ce^{\lambda_k^{\mathfrak{g}_1} r_1}\left\|\psi_k^{\mathfrak{g}_1}- \psi_k^{\mathfrak{g}_2}\right\|_{L^\infty(\Sigma)}^\alpha \left\|\phi_k^{\mathfrak{g}_1}- \phi_k^{\mathfrak{g}_2}\right\|_{W^{1,\infty}(N)}^{1-\alpha}.
\end{equation}
We obtain from the usual interpolation inequalities 
\[
\left\|\phi_k^{\mathfrak{g}_1}- \phi_k^{\mathfrak{g}_2}\right\|_{H^{\frac{3}{2}+s}(N)}\le C\left\|\phi_k^{\mathfrak{g}_1}- \phi_k^{\mathfrak{g}_2}\right\|_{H^2(N)}^{\frac{3+2s}{4}}\left\|\phi_k^{\mathfrak{g}_1}- \phi_k^{\mathfrak{g}_2}\right\|_{L^2(N)}^{\frac{1-2s}{4}}
\]
and hence
\begin{equation}\label{a9}
\|\phi_k^{\mathfrak{g}_1}- \phi_k^{\mathfrak{g}_2}\|_{H^{\frac{3}{2}+s}(N)}\le Ck^{\frac{3+2s}{2n}}\|\phi_k^{\mathfrak{g}_1}- \phi_k^{\mathfrak{g}_2}\|_{L^2(N)}^{\frac{1-2s}{4}}.
\end{equation}
Combining \eqref{a8}, \eqref{a9} and the Weyl's asymptotic formula we find
\[
\left\|\phi_k^{\mathfrak{g}_1}- \phi_k^{\mathfrak{g}_2}\right\|_{H^{\frac{3}{2}+s}(N)}\le Ce^{ck^{\frac{2}{n}}}\left\|\phi_k^{\mathfrak{g}_1}- \phi_k^{\mathfrak{g}_2}\right\|_{W^{1,\infty}(N)}^{\frac{(1-\alpha)(1-2s)}{4}}\left\|\psi_k^{\mathfrak{g}_1}- \psi_k^{\mathfrak{g}_2}\right\|_{L^\infty(\Sigma)}^{\frac{\alpha(1-2s)}{4}}.
\]
This and Theorem \ref{theorem2.1} in Appendix \ref{A} yield
\begin{equation}\label{a10.0}
\left\|\phi_k^{\mathfrak{g}_1}- \phi_k^{\mathfrak{g}_2}\right\|_{H^{\frac{3}{2}+s}(N)}\le Ce^{ck^{\frac{2}{n}}}\left\|\psi_k^{\mathfrak{g}_1}- \psi_k^{\mathfrak{g}_2}\right\|_{L^\infty(\Sigma)}^{\frac{\alpha(1-2s)}{4}}.
\end{equation}
Using the continuity of the trace operator $\mathfrak{t}$ defined in \eqref{to}, we get from \eqref{a10.0}
\begin{equation}\label{a10}
\left\|\psi_k^{\mathfrak{g}_1}- \psi_k^{\mathfrak{g}_2}\right\|_{L^2(\Gamma)}\le Ce^{ck^{\frac{2}{n}}}\left\|\psi_k^{\mathfrak{g}_1}- \psi_k^{\mathfrak{g}_2}\right\|_{L^\infty(\Sigma)}^{\frac{\alpha(1-2s)}{4}}.
\end{equation}

Next, assume that $\mathbf{D}=\mathbf{D}(\mathfrak{g}_1,\mathfrak{g}_2)<\infty$ and $\mathbf{D}_0=\mathbf{D}_0(\mathfrak{g}_1,\mathfrak{g}_2)<\infty$. For each integer $\ell\ge 1$ we have from \eqref{a10}
\[
\sum_{k\le \ell}\alpha_k\left\|\psi_k^{\mathfrak{g}_1}-\psi_k^{\mathfrak{g}_2}\right\|_{L^2(\partial \Omega)}\le Ce^{c\ell^{\frac{2}{n}}}\sum_{k\le \ell}\alpha_k\left\|\psi_k^{\mathfrak{g}_1}-\psi_k^{\mathfrak{g}_2}\right\|_{L^2(\Sigma)}^{\frac{\alpha(1-2s)}{4}}
\]
and hence
\[
\sum_{k\le \ell}\alpha_k\left\|\psi_k^{\mathfrak{g}_1}-\psi_k^{\mathfrak{g}_2}\right\|_{L^2(\partial \Omega)}\le Ce^{c\ell^{\frac{2}{n}}}\sum_{k\ge 1}\alpha_k\left\|\psi_k^{\mathfrak{g}_1}-\psi_k^{\mathfrak{g}_2}\right\|_{L^\infty(\Sigma)}^{\frac{\alpha(1-2s)}{4}}.
\]
This inequality, combined with the following one
\[
\sum_{k> \ell}\alpha_k\left\|\psi_k^{\mathfrak{g}_1}-\psi_k^{\mathfrak{g}_2}\right\|_{L^2(\partial \Omega)}\le \ell^{-\frac{1-2s}{2n}}\sum_{k\ge 1}\left\|\psi_k^{\mathfrak{g}_1}-\psi_k^{\mathfrak{g}_2}\right\|_{L^2(\partial \Omega)},
\]
yields
\begin{equation}\label{a11}
\mathfrak{D}=\sum_{k\ge 1}\alpha_k\left\|\psi_k^{\mathfrak{g}_1}-\psi_k^{\mathfrak{g}_2}\right\|_{L^2(\partial \Omega)}\le  Ce^{c\ell^{\frac{2}{n}}}\mathbf{D}+\ell^{-\frac{1-2s}{2n}}\mathbf{D}_0,
\end{equation}
where $\mathfrak{D}=\mathfrak{D}(\mathfrak{g}_1,\mathfrak{g}_2)=\mathfrak{D}^+(\mathfrak{g}_1,\mathfrak{g}_2)$ is defined Section \ref{subsection1.1}.

Pick $t>0$ and let $\ell=[t^{\frac{n}{2}}]+1$, where $[t^{\frac{n}{2}}]$ is the entire part of $t^{n/2}$. Then it follows from \eqref{a11}
\begin{equation}\label{a12}
C\mathfrak{D}\le e^{ct}\mathbf{D}+t^{-\eta}\mathbf{D}_0,
\end{equation}
where $\displaystyle \eta =\frac{1-2s}{4}$. If $\displaystyle \frac{\mathbf{D}_0}{\mathbf{D}}\le 1$ then \eqref{a12} with $t=1$ gives
\begin{equation}\label{a13}
\mathfrak{D}\le C\mathbf{D}.
\end{equation}
Suppose now that $\displaystyle \frac{\mathbf{D}_0}{\mathbf{D}}> 1$. We then  choose $t$ in \eqref{a12} is such a way that $\displaystyle t^{\eta}e^{ct}= \frac{\mathbf{D}_0}{\mathbf{D}}$. In that case $\displaystyle e^{c't}>\frac{\mathbf{D}_0}{\mathbf{D}}$ and hence $\displaystyle c't>\ln \frac{\mathbf{D}_0}{\mathbf{D}}$.  This inequality in \eqref{a12} gives 
\begin{equation}\label{a14}
\mathfrak{D}\le C \left[\ln \frac{\mathbf{D}_0}{\mathbf{D}}\right]^{-\eta}.
\end{equation}
In light of \eqref{a13} and \eqref{a14} we can assert that
\[
\mathfrak{D}\le C\Psi \left(\frac{\mathbf{D}_0}{\mathbf{D}}\right).
\]
The proof is completed by applying Theorem \ref{theoremCSD}.

\subsection{Proof of Theorem \ref{theoremU}}

We reuse the notations and the assumptions of Subsection \ref{subsection1.2}.

Let $\mathfrak{g}_1,\mathfrak{g}_2\in \mathscr{G}'$ satisfying $\mathfrak{g}_1=\mathfrak{g}_2$ in $N$. Fix $p>n$ and pick $f\in W^{2-\frac{1}{p},p}(\partial \Omega)$ satisfying $\|f\|_{W^{2-\frac{1}{p},p}(\partial \Omega)}\le 1$. From the usual a priori estimate in $H^2$ together with \cite[Theorem 2.3.3.6]{Gr} we get
\begin{equation}\label{dn0}
\|u_{\mathfrak{g}_j}(f)\|_{W^{2,p}(\Omega)}\le C_0,\quad j=1,2,
\end{equation}
where the constant $C_0>0$ only depends on $n$, $\Omega$, $\mathfrak{g}'$, $\mathbf{m}$ and $p$.
As $W^{2,p}(\Omega)$ is continuously embedded in $C^{1,1-\frac{n}{p}}(\overline{\Omega})$, \eqref{dn0} implies
\begin{equation}\label{dn1}
\|u_{\mathfrak{g}_j}(f)\|_{C^{1,1-\frac{n}{p}}(\overline{\Omega})}\le C_0,\quad j=1,2.
\end{equation}

Also, using that $u=u_{\mathfrak{g}_1}(f)-u_{\mathfrak{g}_2}(f)$ is the solution of the BVP \eqref{bvp1} with $\mathfrak{g}=\mathfrak{g}_1$ and $\lambda=0$, we obtain from \eqref{a7}
\begin{equation}\label{dn2}
\|u\|_{L^2(N)}\le C_1\|\partial_\nu u\|_{L^\infty (\Sigma)}^\alpha\|u\|_{W^{1,\infty}(N)}^{1-\alpha},
\end{equation}
where $C_1>0$ and $0<\alpha<1$ only depend on $n$, $\Sigma$, $N$, $\mathfrak{g}'$ and $\mathbf{m}$. Then, inequality \eqref{dn1} in \eqref{dn2} gives
\begin{equation}\label{dn3}
\|u\|_{L^2(N)}\le C_2\|\partial_\nu u\|_{L^\infty (\Sigma)}^\alpha,
\end{equation}
where $C_2>0$  only depends on $n$, $N$, $\Sigma$, $\mathfrak{g}'$, $\mathbf{m}$ and $p$, and $\alpha$ is as in \eqref{dn2}.

In light of the following interpolation inequality 
\[
\|u\|_{H^{3/2+s}(N)}\le  \varkappa \|u\|_{H^2(N)}^{\frac{3+2s}{4}}\|u\|_{L^2(N)}^{\frac{1-2s}{4}},
\]
where $\varkappa >0$ depends only on $n$, $N$ and $\epsilon$, \eqref{dn0}, \eqref{dn3} and the fact that $W^{2,p}(\Omega)$ is continuously embedded in $H^2(\Omega)$, we obtain
\begin{equation}\label{dn4}
\|u\|_{H^{\frac{3}{2}+s}(N)} \le C\|\partial_\nu u\|_{L^\infty (\Sigma)}^{\frac{\alpha(1-2s)}{4}},
\end{equation}
Here and until the end of this subsection $C>0$ denotes a generic constant  depending only on $n$, $N$, $\Sigma$, $\mathfrak{g}'$, $\mathbf{m}$, $s$ and $p$, and $\alpha$ is as in \eqref{dn2}.

Finally, \eqref{dn4} combined with the trace operator $\mathfrak{t}$ defined in \eqref{to} yields
\[
\|\partial_\nu u\|_{L^2(\Gamma )} \le C\|\partial_\nu u\|_{L^\infty (\Sigma)}^\beta,
\]
where $\displaystyle \beta=\frac{\alpha(1-2s)}{4}$.

Let $\mathbf{B}$ denotes the unit ball $W^{2-\frac{1}{p},p}(\partial \Omega)$. As $f$ was taken arbitrary in $\mathbf{B}$ we end up getting
\begin{equation}\label{dn5}
\| \Lambda_{\mathfrak{g}_1}-\Lambda_{\mathfrak{g}_2}\|_{\mathscr{B}(W^{2-\frac{1}{p},p}(\partial \Omega), L^2(\partial \Omega))}\le C\sup_{f\in \mathbf{B}}\| \Lambda_{\mathfrak{g}_1}(f){_{|\Sigma}}-\Lambda_{\mathfrak{g}_2}(f){_{|\Sigma}}\|_{L^\infty (\Sigma)}^\beta .
\end{equation}

Theorem \ref{theoremU} follows then by putting together Theorem \ref{theoremCS} and \eqref{dn5}.

\subsection{Proof of Theorem \ref{theoremPSD2}}

We use the same notations and assumptions as in Subsection \ref{subsection1.3}. For fixed $\sigma >n$, let $\displaystyle \xi_n=\xi_n(\sigma)=1-\frac{n}{\sigma}$ if $n\ne 5$, $\displaystyle \xi_5=\frac{1}{2}$ and
\[
\gamma_n=
\left\{
\begin{array}{lll}
2, &n=3,4,5,
\\
k+1,\quad &n=4k,\; k\ge 2\quad \mathrm{or}\quad n=2k+1,\; k\ge 3,
\\
k+2 &n=4k+2,\; k\ge 2.
\end{array}
\right.
\]

The following result is contained in  \cite[Theorem A.2]{BCKPS}.

\begin{theorem}\label{theorem2.2}
For any $0<\epsilon <1$, $\mathfrak{g}\in \mathscr{G}'$, $\lambda \ge 0$ and $u\in H^2(\mathcal{N})$ satisfying $u=0$ in $\partial \mathcal{M}$, we have
\begin{align*}
C\|\partial _\nu u\|_{L^2(\partial \mathcal{N} )}\le \epsilon ^\alpha(1+\lambda)&\|u\|_{C^{1,\xi_n}(\mathcal{N})}
\\ 
&+e^{\frac{\varkappa}{\epsilon} } \left( \|(\Delta_{\mathfrak{g}} +\lambda )u\|_{L^2(\mathcal{N})}+\|\partial _\nu u\|_{L^2(\Gamma )}\right),
\end{align*}
where the constant $C>0$, $\alpha$ and $\varkappa$ only depends on $n$, $\sigma$, $\mathcal{M}$, $\mathcal{N}$, $\mathfrak{g}'$, $\mathbf{c}$ and $\Gamma_0$.
\end{theorem}

In the sequel $C>0$ will denote a generic constant depending only on $n$, $\sigma$, $\mathcal{M}$, $\mathcal{N}$, $\mathfrak{g}'$, $\mathbf{c}$, $\mathbf{C}_0$ and $\Gamma_0$. Let $\mathfrak{g}_j=c_j\mathfrak{g}_0\in \mathscr{G}_1$, $j=1,2$, with $c_1=c_2$ in $\mathcal{N}$. 
Assume that  
$\mathscr{D}_0=\mathscr{D}_0(\mathfrak{g}_1,\mathfrak{g}_2)<\infty$ and  $\tilde{\mathscr{D}}_0=\tilde{\mathscr{D}}_0(\mathfrak{g}_1,\mathfrak{g}_2)\le \mathbf{C}_0$. We easily check that $u_k=\phi_k^{\mathfrak{g}_1}-\phi_k^{\mathfrak{g}_2}\in H_0^1(\mathcal{M})$, $k\ge 1$, satisfies
\[
\Delta_{\mathfrak{g}_1}u_k+\lambda_k^{\mathfrak{g}_1}u_k=(\lambda_k^{\mathfrak{g}_2}-\lambda_k^{\mathfrak{g}_1})\phi_k^{\mathfrak{g}_2}\quad \rm{in}\; \mathcal{N}.
\]
Let $0<\epsilon<1$. Applying Theorem \ref{theorem2.1} (Appendix \ref{A}) and  Theorem \ref{theorem2.2} with $u=u_k$, we find
\begin{align*}
C\left\|\psi_k^{\mathfrak{g}_1}-\psi_k^{\mathfrak{g}_2}\right\|_{L^2(\partial \mathcal{M} )}&\le \epsilon^\alpha k^{\theta}
\\ 
&+e^{\frac{\varkappa}{\epsilon}} \left( \left|\lambda_k^{\mathfrak{g}_2}-\lambda_k^{\mathfrak{g}_1}\right|+\left\|\psi_k^{\mathfrak{g}_1}-\psi_k^{\mathfrak{g}_2}\right\|_{L^2(\Gamma_0 )}\right),
\end{align*}
where $\displaystyle \theta=\theta_n=\frac{2\gamma_n}{n}$. Therefore for every $\ell> 1$ we get
\begin{align*}
&C\sum_{k\le \ell}\left\|\psi_k^{\mathfrak{g}_1}-\psi_k^{\mathfrak{g}_2}\right\|_{L^2(\partial \mathcal{M} )}\le \epsilon^\alpha \ell^{1+\theta}
\\ 
&\hskip 3.5cm +e^{\frac{\varkappa}{\epsilon}} \sum_{k\le \ell}\left( \left|\lambda_k^{\mathfrak{g}_2}-\lambda_k^{\mathfrak{g}_1}\right|+\left\|\psi_k^{\mathfrak{g}_1}-\psi_k^{\mathfrak{g}_2}\right\|_{L^2(\Gamma_0 )}\right)
\end{align*}
and hence
\begin{equation}\label{z1}
C\sum_{k\le \ell}\left\|\psi_k^{\mathfrak{g}_1}-\psi_k^{\mathfrak{g}_2}\right\|_{L^2(\partial \mathcal{M} )}\le \epsilon^\alpha \ell^{1+\theta}
+e^{\frac{\varkappa}{\epsilon}} \mathscr{D}_0.
\end{equation}
On the other hand we have
\begin{equation}\label{z2}
\sum_{k> \ell}\left\|\psi_k^{\mathfrak{g}_1}-\psi_k^{\mathfrak{g}_2}\right\|_{L^2(\partial \mathcal{M} )}\le \ell^{-\delta}\sum_{k> \ell}k^{\delta}\left\|\psi_k^{\mathfrak{g}_1}-\psi_k^{\mathfrak{g}_2}\right\|_{L^2(\partial \mathcal{M} )}\le \ell^{-\delta}\mathbf{C}_0.
\end{equation}
Putting together \eqref{z1} and \eqref{z2} we get
\begin{equation}\label{z3}
C\mathscr{D}\le \epsilon^\alpha \ell^{1+\theta}
+e^{\frac{\varkappa}{\epsilon}} \mathscr{D}_0 +\ell^{-\delta}.
\end{equation}
Taking $\epsilon=\ell^{-\zeta}$, with $\zeta=(1+\theta+\delta)\alpha$, in this inequality we obtain 
\begin{equation}\label{z4}
C\mathscr{D}\le 
e^{\varkappa \ell^\zeta} \mathscr{D}_0 +\ell^{-\delta}.
\end{equation}
If $t>0$ then $\ell =[t^{\frac{1}{\zeta}}]+1$  in \eqref{z4} yields, upon modifying if necessary $\varkappa$,
\[
C\mathscr{D}\le 
e^{\varkappa t} \mathscr{D}_0 +t^{-\sigma},
\]
where $\displaystyle \sigma=\frac{\delta}{\zeta}$.
We complete the proof similarly to that of Theorem \ref{theoremPSD}.

\section{Stability in relationship with the hyperbolic Dirichlet-to-Neumann map}

We include in this section manifolds of dimension two. We begin by defining the hyperbolic Dirichlet-to-Neumann map. For $\mathfrak{g}\in \mathscr{G}$ and $T>0$ consider the following IBVP 
\begin{equation}\label{w1}
\left\{
\begin{array}{ll}
(\partial_t^2-\Delta_\mathfrak{g})u=0\quad \mathrm{in}\; (0,T)\times \mathcal{M},
\\
u(0,\cdot)=\partial_tu(0,\cdot)=0,
\\
u_{|(0,T)\times \partial \mathcal{M}}=\varphi
\end{array}
\right.
\end{equation}
and set
\[
H_{0,}^1((0,T)\times \partial \mathcal{M})=\{\varphi \in H^1((0,T)\times \partial \mathcal{M});\; \varphi(0,\cdot)=0\}.
\]

Let $\varphi \in H_{0,}^1((0,T)\times \partial \mathcal{M})$.  From \cite[Theorem 2.30]{KKL} the IBVP \eqref{w1} admits a unique solution 
\[
u^{\mathfrak{g}}(\varphi)\in C^1([0,T],L^2(\mathcal{M}))\cap C([0,T],H^1(\mathcal{M})).
\]
 According to \cite[3.2.1, 3.2.2]{KKL} this solution is given by the formula
\begin{equation}\label{w2}
u^{\mathfrak{g}}(\varphi)(t)=\sum_{k\ge 1}\left(\int_0^t \langle \varphi(t-s,\cdot)|\psi_k^{\mathfrak{g}}\rangle_{\mathfrak{g}}s_k^{\mathfrak{g}}(s)ds\right) \phi_k^{\mathfrak{g}},\quad t\in [0,T],
\end{equation}
where 
\[
s_k^{\mathfrak{g}}(t)=\frac{\sin\left( \sqrt{\lambda_k^{\mathfrak{g}}}\, t\right)}{\sqrt{\lambda_k^{\mathfrak{g}}}},\quad t\in [0,T].
\]
Furthermore we have $\partial_\nu u^{\mathfrak{g}}(\varphi)\in L^2((0,T)\times \partial \mathcal{M})$ (hidden regularity) and
\[
\|\partial_\nu u^{\mathfrak{g}}(\varphi)\|_{L^2((0,T)\times \partial \mathcal{M})}\le C_0\|\varphi\|_{H^1((0,T)\times \partial \mathcal{M})},
\]
where the constant $C_0>0$ depends only on $n$, $\mathcal{M}$ and $\mathfrak{g}$.

Let $\Pi_{\mathfrak{g}}\in \mathscr{B}\left(H_{0,}^1((0,T)\times \partial \mathcal{M}),L^2((0,T)\times \partial \mathcal{M})\right)$ be the hyperbolic Dirichlet-to-Neumann map which is defined as follows
\[
\Pi_{\mathfrak{g}}(\varphi)= \partial_\nu u^{\mathfrak{g}}(\varphi),\quad \varphi \in H_{0,}^1((0,T)\times \partial \mathcal{M}).
\]
Pick  $\mathfrak{g}_1,\mathfrak{g}_2\in \mathscr{G}_0$, $\varphi\in C_0^\infty((0,T)\times \partial \mathcal{M})$ and set
\[
a_k^{\mathfrak{g}_j}(t)=\left(\int_0^t \left\langle \varphi(t-s,\cdot)|\psi_k^{\mathfrak{g}_j}\right\rangle s_k^{\mathfrak{g}_j}(s)ds\right) \psi_k^{\mathfrak{g}_j},\quad t\in [0,T],\; j=1,2.
\]
Fix an integer $\ell \ge 1$ and let
\[
w(t)=\sum_{k\le \ell}\left[ a_k^{\mathfrak{g}_1}(t)-a_k^{\mathfrak{g}_2}(t)\right],\quad t\in [0,T].
\]
We split $w$ into three terms: $w=w_1+w_2+w_3$ where 
\begin{align*}
&w_1(t)=\sum_{k\le \ell}\left(\int_0^t \left\langle \varphi(t-s,\cdot)|\psi_k^{\mathfrak{g}_1}-\psi_k^{\mathfrak{g}_2}\right\rangle s_k^{\mathfrak{g}_1}(s)ds\right) \psi_k^{\mathfrak{g}_1},
\\
&w_2(t)= \sum_{k\le \ell}\left(\int_0^t \left\langle \varphi(t-s,\cdot)|\psi_k^{\mathfrak{g}_2}\right\rangle \left(s_k^{\mathfrak{g}_1}(s)-s_k^{\mathfrak{g}_1}(s)\right)ds\right) \psi_k^{\mathfrak{g}_1},
\\
&w_3(t)=\sum_{k\le \ell}\left(\int_0^t \left\langle \varphi(t-s,\cdot)|\psi_k^{\mathfrak{g}_2}\right\rangle s_k^{\mathfrak{g}_2}(s)ds\right)\left(\psi_k^{\mathfrak{g}_1}-\psi_k^{\mathfrak{g}_2}\right).
\end{align*}
We have
\begin{align*}
&\left|\int_0^t \left\langle \varphi(t-s,\cdot)|\psi_k^{\mathfrak{g}_1}-\psi_k^{\mathfrak{g}_2}\right\rangle s_k^{\mathfrak{g}_1}(s)ds\right|
\\
&\hskip 2cm \le \left(\int_0^t \|\varphi(t-s,\cdot)\|_{L^2(\partial \mathcal{M})}\left|s_k^{\mathfrak{g}_1}(s)\right|ds\right)\left\|\psi_k^{\mathfrak{g}_1}-\psi_k^{\mathfrak{g}_2}\right\|_{L^2(\partial \mathcal{M})}
\\
&\hskip 2cm \le \sqrt{T}\left\|s_k^{\mathfrak{g}_1}\right\|_{L^\infty((0,T))}\left\|\psi_k^{\mathfrak{g}_1}-\psi_k^{\mathfrak{g}_2}\right\|_{L^2(\partial \mathcal{M})}\|\varphi\|_{L^2((0,T)\times\partial \mathcal{M})}
\\
&\hskip 2cm \le \frac{\sqrt{T}}{\sqrt{\lambda_k^{\mathfrak{g}}}}\left\|\psi_k^{\mathfrak{g}_1}-\psi_k^{\mathfrak{g}_2}\right\|_{L^2(\partial \mathcal{M})}\|\varphi\|_{L^2((0,T)\times\partial \mathcal{M})}.
\end{align*}
Since
\[
\|w_1(t)\|_{L^2(\partial \mathcal{M})}\le \sum_{k\le \ell}\left|\int_0^t \langle \varphi(t-s,\cdot)|\psi_k^{\mathfrak{g}_1}-\psi_k^{\mathfrak{g}_2}\rangle s_k^{\mathfrak{g}_1}(s)ds\right| \left\|\psi_k^{\mathfrak{g}_1}\right\|_{L^2(\partial \mathcal{M})},
\]
we obtain from the last inequality and \eqref{e3} 
\[
\|w_1\|_{L^2((0,T)\times \partial \mathcal{M})}\le CT\sum_{k\le \ell}k^{\frac{1+2s}{2n}}\left\|\psi_k^{\mathfrak{g}_1}-\psi_k^{\mathfrak{g}_2}\right\|_{L^2(\partial \mathcal{M})}\|\varphi\|_{L^2((0,T)\times\partial \mathcal{M})}.
\]
Here and henceforth $C>0$ is a generic constant depending only on $n$, $\mathcal{M}$, $s$, $\mathfrak{g}'$, $\mathfrak{g}_0$ and $\mathbf{m}$. On the other hand, using that
\[
\left| s_k^{\mathfrak{g}_j}(t)-s_k^{\mathfrak{g}_j}(t)\right|\le \frac{\left|\lambda_k^{\mathfrak{g}_1}-\lambda_k^{\mathfrak{g}_2}\right|}{\sqrt{\lambda_k^{\mathfrak{g}_1}}\sqrt{\lambda_k^{\mathfrak{g}_2}}\left(\sqrt{\lambda_k^{\mathfrak{g}_1}}+\sqrt{\lambda_k^{\mathfrak{g}_2}}\right)}
+\frac{T\left|\lambda_k^{\mathfrak{g}_1}-\lambda_k^{\mathfrak{g}_2}\right|}{\sqrt{\lambda_k^{\mathfrak{g}_2}}\left(\sqrt{\lambda_k^{\mathfrak{g}_1}}+\sqrt{\lambda_k^{\mathfrak{g}_2}}\right)},
\]
we obtain
\[
\|w_2\|_{L^2((0,T)\times \partial \mathcal{M})}\le C\sqrt{T}(1+T) \sum_{k\le \ell}k^{-\frac{1-2s}{n}}\left|\lambda_k^{\mathfrak{g}_1}-\lambda_k^{\mathfrak{g}_2}\right|\|\varphi\|_{L^2((0,T)\times \partial \mathcal{M})}.
\]
Also, we check that $w_3$ satisfies the same estimate as $w_1$. In summary we obtain that the series $\displaystyle \sum_{k\ge 1}\left[ a_k^{\mathfrak{g}_1}-a_k^{\mathfrak{g}_2}\right]$ converges in $L^2((0,T)\times \partial \mathcal{M})$ and
\begin{equation}\label{w6}
\left\|\sum_{k\ge 1}\left[ a_k^{\mathfrak{g}_1}-a_k^{\mathfrak{g}_2}\right]\right\|_{L^2((0,T)\times \partial \mathcal{M})}\le C(1+\sqrt{T})(1+T) \mathbf{d}\|\varphi\|_{L^2((0,T)\times \partial \mathcal{M})}
\end{equation}
provided that
\[
\mathbf{d}=\mathbf{d}(\mathfrak{g}_1,\mathfrak{g}_2)=\sum_{k\ge 1}\left[\alpha_k^2\left|\lambda_k^{\mathfrak{g}_1}-\lambda_k^{\mathfrak{g}_2}\right|+\sqrt{\beta_k}\left\|\psi_k^{\mathfrak{g}_1}-\psi_k^{\mathfrak{g}_2}\right\|_{L^2(\partial \mathcal{M})}\right]<\infty,
\]
where the sequences $(\alpha_k)$ and $(\beta_k)$ is given in \eqref{se}.

Noting that $C_0^\infty ((0,T)\times \partial\mathcal{M})$ is dense in $L^2((0,T)\times \partial\mathcal{M})$, we derive that \eqref{w6} holds for any $\varphi \in L^2((0,T)\times \partial\mathcal{M})$. 
In other words, $\Pi_{\mathfrak{g}_1}-\Pi_{\mathfrak{g}_2}$ extends to bounded operator from $L^2((0,T)\times \partial\mathcal{M})$ into $L^2((0,T)\times \partial\mathcal{M})$ with
\begin{equation}\label{w7}
\left\| \Pi_{\mathfrak{g}_1}-\Pi_{\mathfrak{g}_2}  \right\|_{\mathscr{B}(L^2((0,T)\times \partial\mathcal{M}))}\le C(1+\sqrt{T})(1+T) \mathbf{d}.
\end{equation}

Assume that $\mathcal{M}=\overline{\Omega}$ and denote by $\mathbf{G}_s$ the set of metrics $\mathfrak{g}\in \mathscr{G}$ so that $(\overline{\Omega}, \mathfrak{g})$ is simple (see the definition in Subsection \ref{subsection1.1}).

The problem of determining a metric tensor $\mathfrak{g}\in \mathscr{G}$ from the corresponding distance function $d_{\mathfrak{g}}$ is known as the boundary rigidity problem. In \cite{SU3} the authors show that for $p\gg 1$ there exists $\mathbf{G}^p\subset \mathbf{G}_s$ dense in $C^p(\overline{\Omega})$ so that any $\mathfrak{g}\in \mathbf{G}^k$ is uniquely determined, up to isometry which is the identity on the boundary.

\begin{theorem}\label{theoremSDNH}
$($\cite{SU2}$)$ There exist $p\ge 2$ and $0<\mu <1$, such that for any $\tilde{\mathfrak{g}}\in \mathbf{G}^p$, $T>\mathrm{diam}_{\tilde{\mathfrak{g}}}(\Omega)$ and $0<t_0 <T-\mathrm{diam}_{\tilde{\mathfrak{g}}}(\Omega)$, there exist $\eta >0$ and $\mathbf{r}$ with the property that if $\mathfrak{g}_1,\mathfrak{g}_2\in \mathscr{G}$ satisfying
\[
\| \mathfrak{g}_j-\tilde{\mathfrak{g}}\|_{C(\overline{\Omega})}\le \eta,\quad \| \mathfrak{g}_j\|_{C^p(\overline{\Omega})}\le \mathbf{r},\quad j=1,2,
\]
then one can find a $C^3$-diffeomophism $F:\overline{\Omega}\rightarrow \overline{\Omega}$ satisfying $F_{|\partial \Omega}=I$ such that
\begin{equation}\label{w8}
\|\mathfrak{g}_1-F^\ast \mathfrak{g}_2\|_{C^2(\overline{\Omega})}\le C'\left\| \Pi_{\mathfrak{g}_1}-\Pi_{\mathfrak{g}_2}  \right\|_{\mathscr{B}(H_0^1((0,t_0) \times \partial \Omega)), L^2((0,T)\times \partial\mathcal{M}))}^\mu, 
\end{equation}
where $C'>0$ only depends on $n$, $\Omega$, $T$, $t_0$, $\tilde{\mathfrak{g}}$, $p$, $\mu$, $\eta$ and $\mathbf{r}$.
\end{theorem}

This theorem was generalized to the magnetic Laplace-Beltrami operator in \cite{Mo}. The case of metric tensors close to the euclidian metric was considered in \cite{SU1}. A recent logarithmic stability inequality in the case where  the metric tensor is assumed to be known in a neighborhood of the boundary was proved in \cite{Bell}. This result holds only by assuming the knowledge of the hyperbolic Dirichlet-to-Neumann in an arbitrary subboundary for $T$ is sufficiently large. The reader is also refered to \cite{SUV} where the case of incomplete boundary data was discussed.

Let $p\ge 2$ and $0<\mu <1$ as in Theorem \ref{theoremSDNH}. Fix $\tilde{\mathfrak{g}}\in \mathbf{G}^p$ and let $T$, $t_0$, $\eta$ and  $\mathbf{r}$, corresponding to that $\tilde{\mathfrak{g}}$ as they appear in Theorem \ref{theoremSDNH}. Define then
\begin{align*}
\tilde{\mathbf{G}}=\{(\mathfrak{g}_1,\mathfrak{g}_2)\in \mathscr{G}\times \mathscr{G};\; \mathfrak{g}_1{_{|\Gamma}}&=\mathfrak{g}_2{_{|\Gamma}},
\\
& \| \mathfrak{g}_j-\tilde{\mathfrak{g}}\|_{C(\overline{\Omega})}\le \eta,\quad \| \mathfrak{g}_j\|_{C^p(\overline{\Omega})}\le \mathbf{r},\; j=1,2\}.
\end{align*}

Combining \eqref{w7} and \eqref{w8}, we get the following stability estimate.

\begin{theorem}\label{theoremHS}
Let $(\mathfrak{g}_1,\mathfrak{g}_2)\in \tilde{\mathbf{G}}$ satisfying $\mathbf{d}=\mathbf{d}(\mathfrak{g}_1,\mathfrak{g}_2)<\infty$. Then one can find a $C^3$-diffeomophism $F:\overline{\Omega}\rightarrow \overline{\Omega}$ satisfying $F_{|\partial \Omega}=I$ such that
\begin{equation}\label{w9}
\|\mathfrak{g}_1-F^\ast \mathfrak{g}_2\|_{C^2(\overline{\Omega})}\le C''\mathbf{d}^\mu, 
\end{equation}
where $C''>0$ only depends on $n$, $\Omega$, $T$, $t_0$, $s$, $\tilde{\mathfrak{g}}$, $p$, $\mu$, $\eta$ and $\mathbf{r}$.
\end{theorem} 

It is worth noticing that, within the class of metric tensors introduced in \cite[Definition 1.1]{Bell},  one can establish with the help of \cite[Theorem 2.2]{Bell} a logarithmic stability inequality in the case of partial boundary spectral data in an open subset of $\Gamma$. We leave to the reader to write down the details.

\appendix

\section{H\"older a priori estimate for eigenfunctions}\label{A}

Recall that $\sigma >n$ is fixed,  $\xi_n=\xi_n(\sigma)=1-\frac{n}{\sigma}$ if $n\ne 5$, $p_5=1/2$ and
\[
\gamma_n=
\left\{
\begin{array}{lll}
2, &n=3,4,5,
\\
k+1,\quad &n=4k,\; k\ge 2\quad \mathrm{or}\quad n=2k+1,\; k\ge 3,
\\
k+2 &n=4k+2,\; k\ge 2.
\end{array}
\right.
\]

\begin{theorem}\label{theorem2.1}
Let $\mathfrak{g}\in \mathscr{G}'$, $\lambda \in \sigma (A_{\mathfrak{g}})$ and let $u$ be an eigenfunction corresponding to the eingenvalue $\lambda$ so that $\|u\|_{L^2(\mathcal{M})}=1$. Then
\begin{equation}\label{2.0.1}
\|u\|_{C^{1,p_n}(\mathcal{M})}\le C\lambda^{\xi_n},
\end{equation}
where the constant $C>0$ only depends of $n$, $\mathcal{M}$, $\mathfrak{g}'$ $\mathbf{m}$, and $\sigma$.
\end{theorem}

\begin{proof}
 In this proof, $C>0$ denotes a generic constant only depending of $n$, $\mathcal{M}$, $\mathfrak{g}'$, $\mathbf{m}$ and $\sigma$.  We already know that
\begin{equation}\label{2.0}
\|u\|_{H^2(\mathcal{M})}\le C\lambda .
\end{equation}
 
Let $n=3$. As $H^2(\mathcal{M})$ is continuously embedded in $L^\infty (\mathcal{M})$, \eqref{2.0} yields
\[
\|u\|_{L^\sigma(\mathcal{M})}\le C\lambda .
\]
Applying \cite[Theorem 9.14 in page 240]{GT} we find
\begin{align*}
\|u\|_{W^{2,\sigma}(\mathcal{M})}&\le C\lambda \|u\|_{L^\sigma(\mathcal{M})}
\\
&\le C\lambda^2.
\end{align*}
This and the fact that $W^{2,\sigma}(\mathrm{M})$ is continuously embedded in $C^{1,1-\frac{3}{\sigma}}(\mathcal{M})$ imply
\[
\|u\|_{C^{1,1-\frac{3}{\sigma}}(\mathcal{M})}\le C\lambda^2.
\]

Next, consider the case $n=4$. For $1<r<2$, since $H^2(\mathrm{M})$ is continuously embedded in $W^{2,r}(\mathrm{M})$, we derive from \eqref{2.0}
\begin{equation}\label{equ7}
\|u\|_{W^{2,r}(\mathcal{M})}\le C\lambda.
\end{equation}
Using that $W^{2,r}(\mathcal{M})$ is continuously embedded in $L^{\hat{r}}(\mathcal{M})$ with $\displaystyle \hat{r}=\frac{r}{2-r}$, we obtain from \eqref{equ7} 
\begin{equation}\label{equ8}
\|u\|_{L^{\hat{r}}(\mathrm{M})}\le C\lambda.
\end{equation}
Again, \cite[Theorem 9.14 in page 240]{GT} yields
\[
\|u\|_{W^{2,\hat{r}}(\mathcal{M})}\le C\lambda \|u\|_{L^{\hat{r}}(\mathcal{M})}
\]
which, combined with \eqref{equ8}, implies
\begin{equation}\label{equ10}
\|u\|_{W^{2,\hat{r}(}\mathcal{M})}\le C\lambda^2.
\end{equation}
We choose $r$ in such a way that $\hat{r}=\sigma$. As $W^{2,\sigma}(\mathcal{M})$ is continuously embedded in $C^{1,1-\frac{4}{\sigma}}(\mathcal{M})$ we find 
\[
\|u\|_{C^{1,1-\frac{4}{\sigma}}(\mathrm{M})}\le C\lambda ^2.
\]

Next, assume that $n>4$. This condition guarantees  that $H^2(\mathcal{M})$ is continuously embedded in $L^{r_1}(\mathrm{M})$ where $\displaystyle r_1=\frac{2n}{n-4}$. Hence \eqref{2.0} yields
\[
\|u\|_{L^{r_1}(\mathcal{M})}\le C\lambda,
\]
from which we derive  as in the preceding case
\begin{equation}\label{equ12}
\|u\|_{W^{2,r_1}(\mathcal{M})}\le C\lambda^2.
\end{equation}

If $n=5$ then $W^{2,10}(\mathcal{M})$ is continuously embedded in $C^{1,\frac{1}{2}}(\mathcal{M})$ and therefore \eqref{equ12} yields
\[
\|u\|_{C^{1,\frac{1}{2}}(\mathcal{M})}\le C\lambda^2.
\]
In the sequel
\[
f(r)=\frac{nr}{n-2r},\quad 2\le 2r<n.
\]

Assume now that $n=4k$, $k\ge 2$. In this case $\displaystyle r_1=\frac{2k}{k-1}=\frac{2n}{n-4}$. Then let
\[
r_{j+1}=f(r_j),\quad 1\le j\le k-1,
\]
Simple calculations give $\displaystyle f(r)=\frac{2kr}{2k-r}$, $1\le r<2k$, from which we obtain
\[
r_j=\frac{2k}{k-j},\quad 1\le j\le k-1.
\]
Suppose that we proved for $1\le j<k-2$ 
\begin{equation}\label{equ12.9}
\|u\|_{W^{2,r_j}(\mathcal{M})}\le C\lambda^{j+1}.
\end{equation}
Using that $W^{2,r_j}(\mathcal{M})$ is continuously embedded in $L^{r_{j+1}}(\mathcal{M})$ we get
\[
\|u\|_{L^{r_{j+1}}(\mathcal{M})}\le C\lambda^{j+1}.
\]
A priori estimate in $W^{2,r_{j+1}}(\mathcal{M})$ then gives
\[
\|u\|_{W^{2,r_{j+1}}(\mathcal{M})}\le C\lambda^{j+2}.
\]
In other words, we proved by induction in $j$ that \eqref{equ12.9} holds.

Observe that, since $2r_{k-1}=4k=n$, we know that $W^{2,r_{k-1}}(\mathcal{M})$ is not necessarily continuously embedded in $L^{r_{k-1}}(\mathcal{M})$. For each $1<s_k<2k$, as $W^{2,s_k}(\mathcal{M})$ is continuously embedded in $W^{2,r_{k-1}}(\mathcal{M})$, it follows from \eqref{equ12.9} that
\[
\|u\|_{W^{2,s_k}(\mathcal{M})}\le C\lambda^{k} 
\]
and hence
\[
\|u\|_{L^{f(s_k)}(\mathcal{M})}\le C\lambda^{k} .
\]
Again, using a priori estimate in $W^{2,f(s_k)}$, we obtain
\begin{equation}\label{equ12.12}
\|u\|_{W^{2,f(s_k)}(\mathcal{M})}\le C\lambda^{k+1}.
\end{equation}
We choose $s_k$ in such a way that $f(s_k)=\sigma$. In that case \eqref{equ12.12} takes the form
\begin{equation}\label{equ12.13}
\|u\|_{W^{2,\sigma}(\mathcal{M})}\le C\lambda^{k+1}.
\end{equation}
Finally, using that $W^{2,\sigma}(\mathcal{M})$ is continuously embedded in $C^{1,1-\frac{n}{\sigma}}(\mathcal{M})$ and \eqref{equ12.13} we obtain
\[
\|u\|_{C^{1,1-\frac{n}{\sigma}}(\mathcal{M})}\le C\lambda^{k+1}.
\]

Consider  now the case $n=4k+2$ with $k\ge 1$. In this case
\[
f(r)=\frac{(2k+1)r}{2k+1-r},\quad r<2k+1\quad  \mbox{and}\quad  r_j=\frac{4k+2}{2(k-j)+1},\quad 1\le j\le k.
\]
In particular we have $r_k=n$ and  
\begin{equation}\label{equ12.9.0}
\|u\|_{W^{2,n}(\mathcal{M})}\le C\lambda^{k+1}.
\end{equation}
We proceed as before by  choosing $1<s_k<4k+2$ so that $f(s_k)=\sigma$. We obtain
\[
\|u\|_{W^{2,\sigma}(\mathcal{M})}\le C\lambda^{k+2}
\]
and hence
\begin{equation}\label{equ12.14.0}
\|u\|_{C^{1,1-\frac{n}{\sigma}}(\mathcal{M})}\le C\lambda^{k+2}.
\end{equation}

It remains to study the case $n=2k+1$, $k\ge 3$, for which we have
\[
r_1=\frac{4k+2}{2k-3},\quad r_j=f(r_{j-1})=\frac{4k+2}{2k-(2j+1)},\quad 2\le j\le k-1.
\]
Here we have \[
f(r)=\frac{(2k+1)r}{(2k+1)-2r},\quad 1\le 2r<2k+1.
\]
As above, since 
\[
\|u\|_{W^{2,r_j}(\mathcal{M})}\le C\lambda^{1+j},\quad 1\le j\le k-1,
\]
taking $j=k-1$, we find
\[
\|u\|_{W^{2,n}(\mathcal{M})}\le C\lambda^k,
\]
from which we derive similarly to the preceding case
\[
\|u\|_{C^{1,1-\frac{n}{\sigma}}(\mathcal{M})}\le C\lambda^{k+1}.
\]
The proof is then complete.
\end{proof}

\end{document}